\documentclass[11pt]{amsart}
\usepackage{hyperref}
\usepackage{amssymb,amsmath,amsthm,mathtools}
\usepackage{enumerate}
\newcommand*{\mailto}[1]{\href{mailto:#1}{\nolinkurl{#1}}}

\newcommand{\doi}[1]{\href{http://dx.doi.org/#1}{doi:#1}}

\newtheorem{theorem}{Theorem}[section]

\newtheorem{lemma}[theorem]{Lemma}
\newtheorem{corollary}[theorem]{Corollary}

\newtheorem{hypothesis}{Hypothesis}

\theoremstyle{definition}
\newtheorem{remark}[theorem]{Remark}
\newtheorem*{acknowledgments}{Acknowledgments}

\newcommand{\R}{{\mathbb R}}
\newcommand{\N}{{\mathbb N}}
\newcommand{\Z}{{\mathbb Z}}
\newcommand{\C}{{\mathbb C}}

\newcommand{\be}{\begin{equation}}
\newcommand{\ee}{\end{equation}}

\newcommand{\ti}{\tilde}

\newcommand{\E}{\mathrm{e}}
\newcommand{\I}{\mathrm{i}}

\newcommand{\tr}{\mathrm{tr}}

\DeclareMathOperator{\lspan}{span}

\DeclareMathOperator{\im}{im}
\DeclareMathOperator{\dom}{dom}

\DeclareMathOperator{\ran}{ran}

\DeclareMathOperator{\assoc}{assoc}

\newcommand{\floor}[1]{\lfloor#1 \rfloor}

\newcommand{\abs}[1]{\left\lvert #1 \right\rvert}
\newcommand{\norm}[1]{\left\lVert #1 \right\rVert}
\newcommand{\inner}[2]{\left\langle#1,#2\right\rangle}

\newcommand{\cH}{\mathcal{H}}
\newcommand{\cB}{\mathcal{B}}

\renewcommand\tilde{\widetilde}
\renewcommand\hat{\widehat}

\newcommand{\eps}{\varepsilon}

\newcommand{\sig}{\sigma}
\newcommand{\lam}{\lambda}

\numberwithin{equation}{section}

\begin{document}

\title[${N}$-entire singular Schr\"odinger operators]{Singular Schr\"odinger operators as self-adjoint extensions of 
	  $\boldsymbol{N}$-entire operators}

\author[L. O. Silva]{Luis O. Silva}
\address{Departamento de F\'{i}sica Matem\'{a}tica\\
Instituto de Investigaciones en Ma{\-}te{\-}m\'{a}{\-}ti{\-}cas Aplicadas
y en Sistemas\\
Universidad Nacional Aut\'{o}noma de M\'{e}xico\\
C.P. 04510, M\'{e}xico D.F.}
\email{\mailto{silva@iimas.unam.mx}}

\author[G. Teschl]{Gerald Teschl}
\address{Faculty of Mathematics\\ University of Vienna\\
Oskar-Morgenstern-Platz 1\\ 1090 Wien\\ Austria\\ and
International Erwin Schr\"odinger Institute for Mathematical Physics\\
 Boltzmanngasse 9\\ 1090 Wien\\ Austria}
\email{\mailto{Gerald.Teschl@univie.ac.at}}
\urladdr{\url{http://www.mat.univie.ac.at/~gerald/}}

\author[J. H. Toloza]{Julio H. Toloza}
\address{CONICET\\ and
Centro de Investigaci\'{o}n en Inform\'{a}tica para la Ingenier\'{i}a\\
Universidad Tecnol\'{o}gica Nacional -- Facultad Regional C\'{o}rdoba\\
Maestro L\'{o}pez s/n\\
X5016ZAA C\'{o}rdoba, Argentina}
\email{\mailto{jtoloza@scdt.frc.utn.edu.ar}}

\thanks{Proc. Amer. Math. Soc. {\bf 143}, 2103--2115 (2015)}
\thanks{\em Research supported by the Austrian Science Fund (FWF) under
Grant No.\ Y330
and by CONICET (Argentina) through grant PIP 112-201101-00245}


\begin{abstract}
  We investigate the connections between Weyl--Titchmarsh--Kodaira
  theory for one-dimensional Schr\"odinger operators and the theory of
  $n$-entire operators.  As our main result we find a necessary and
  sufficient condition for a one-dimensional Schr\"odinger operator to
  be $n$-entire in terms of square integrability of derivatives
  (w.r.t. the spectral parameter) of the Weyl solution. We also show
  that this is equivalent to the Weyl function being in a generalized
  Herglotz--Nevanlinna class. As an application we show that perturbed
  Bessel operators are $n$-entire, improving the previously known
  conditions on the perturbation.
\end{abstract}

\maketitle

\section{Introduction}

The Weyl-Titchmarsh-Kodaira theory for a self-adjoint operator $H$
associated with the differential expression
\begin{equation}
\label{stl0}
\tau := - \frac{d^2}{dx^2} + q(x),\quad -\infty\le a<x<b \le \infty,
\end{equation}
where the potential $q$ is real-valued and satisfies
\begin{equation}
\label{eq:potential}
q \in L^1_{\text{\em loc}}(a,b).
\end{equation}
has been an active and alluring subject of research for long time,
particularly nowadays. The current interest concerns the case where
both endpoints are generically singular. Recent developments show
that, under a necessary and sufficient additional condition on $q(x)$
(see Hypothesis~\ref{hyp:gen} below), there exists an entire system of
fundamental solutions $\phi(z,x),\theta(z,x)$ of the equation
$\tau\varphi=z\varphi$ such that the Wronskian of these two solutions
equals one, and one of the solutions (say $\phi(z,x)$) is in the domain
of $H$ near the left endpoint. A singular Weyl function $M(z)$
(associated with the left endpoint) is then defined as a function that
makes
\begin{equation}
\label{eq:psi-intro}
\psi(z,x):=\theta(z,x)+M(z)\phi(z,x)
\end{equation}
be in the domain of $H$ near the right endpoint (more details are
accounted for in the next section).  As in the regular case, $M(z)$
encodes all the spectral information related to $H$. However, contrary
to the regular case, there is not a natural choice of normalization
for the entire system of fundamental solutions and, therefore, the
singular Weyl function $M(z)$ does not generically belong to a
particular class of functions as in the regular case.

This indeterminacy has been overcome for the class of perturbed
spherical Schr\"odinger operators (also known as Bessel operators), a
class of operators that has attracted considerable interest recently;
see for example
\cite{ahm2,car,car2,bg,je,et,ek,gz,gr,kst,kst2,kst3,kt,lw,lv,se}.  For
this class of operators, and assuming some mild additional conditions
on $q(x)$, a technique for constructing the system of fundamental
solutions $\phi(z,x),\theta(z,x)$ based on Frobenius method has been
proposed as the natural choice. In particular, it has been proved that
in this case, $M(z)$ belongs to a specific class of the generalized
Nevanlinna functions $N_\kappa^\infty$.

In this paper we address the issue of elucidating some properties of
$M(z)$ and the singular Weyl solution $\psi(z,x)$ given in
(\ref{eq:psi-intro}) from a different perspective, although restricted
to the cases where $H$ has only discrete spectrum. Here we consider
$H$ as a self-adjoint extension of some symmetric, regular (hence
completely non-self-adjoint) operator $A$ with deficiency indices
$(1,1)$ (for all the technical definitions see
Section~\ref{sec:n-entire}). Among these operators there exists a
distinguished class, the so-called $n$-entire operators \cite{st1},
defined by the condition
\begin{equation*}
\cH = \ran(A-zI)\dotplus\lspan\{\mu_0+z\mu_1+\cdots+z^n\mu_n\},
\quad z\in\C,
\end{equation*}
for some fixed $\mu_0,\ldots,\mu_n\in\cH$; here $\cH$ stands for the
Hilbert space in which $A$ is defined. As discussed in \cite{st1},
every $n$-entire operator can be unitarily transformed into the
operator of multiplication by the independent variable acting on a de
Branges space $\cB_A$. Moreover, $\cB_A$ is such that the linear
manifold $\assoc_n(\cB_A)$ of $n$-associated functions contains a
zero-free entire function. On the basis of this setup, we establish in
this work an equivalence between (a) the operator $A$ being
$n$-entire, (b) the $(n-1)$-th derivative of the singular Weyl
solution  $\psi(z,x)$ being square integrable with respect
to the spectral measure, and (c) the possibility of choosing the
solution $\theta(z,x)$ such that $M(z)\in N^\infty_{n-1}$. Precise
formulations of this assertion, tied to the fulfillment of an increasing
number of technical conditions, are given in Theorems 
\ref{thm:main-equivalences}, \ref{thm:main-equivalences-true} and
\ref{thm:main-equivalence-complete}. 

Reciprocally, the connection between the Weyl--Titchmarsh--Kodaira
theory for one-dimensional Schr\"odinger operators and the theory of
$n$-entire operators can be exploited in another direction, as it
allows to broaden the classes of differential operators that are known
to be $n$-entire. A first investigation of this matter has been done
in \cite{st2}, where it is shown that perturbed Bessel operators on a
finite interval are $n$-entire, with $n$ given in terms of the angular
momentum number, provided that $q(x)$ obeys a certain rather restrictive
technical condition that arises from the perturbation arguments used
in that paper.  As an application of the results
obtained in this work, we generalize the classes addressed in
\cite{st2} by lifting this technical restriction; this is asserted in
Theorem~\ref{thm:bessel}.

We conclude this introduction with an outline of this paper. All the
relevant aspects of the Weyl--Titchmarsh--Kodaira theory are reviewed
in Section~\ref{sec:swm}. The theory of $n$-entire operators as well
as their connection with the theory of de Branges spaces is briefly
recalled in Section~\ref{sec:n-entire}. Finally,
Section~\ref{sec:main-results} contains the main results of this work.

\section{Singular Weyl--Titchmarsh--Kodaira theory}
\label{sec:swm}

One of our fundamental ingredients will be singular
Weyl--Titchmarsh--Kodaira theory and hence we begin by recalling the
necessary facts from \cite{kst2}.  Consider one-dimensional
Schr\"odinger operators on $L^2(a,b)$ with $-\infty \le a<b \le
\infty$ associated with the differential expression (\ref{stl0}) with
potential (\ref{eq:potential}). We use $H$ to denote a self-adjoint
operator given by $\tau$ with separated boundary conditions at $a$
and/or $b$.  For further background we refer to
\cite[Chap.~9]{tschroe} or \cite{wdln}.

As mentioned in the Introduction, to define the singular Weyl function
at the, in general singular, endpoint $a$, we need a fundamental
system of solutions $\theta(z,x)$ and $\phi(z,x)$ of the equation
$\tau\varphi=z\varphi$ which are entire with respect to $z$ and such
that $\phi(z,x)$ lies in the domain of $H$ near $a$ and the Wronskian
\begin{equation}
  \label{eq:wronskian}
 W(\theta(z),\phi(z)):= \theta(z,x)
\phi'(z,x) - \theta'(z,x)\phi(z,x)\equiv 1.
\end{equation}
Recall that the Wronskian does not depend on $x$ when its arguments
are solutions of the same equation. Thus, (\ref{eq:wronskian}) tells
us that the function of $z$ on the l.\,h.\,s is identically 1.

Denote the restriction of $H$ to $(a,c)$ with a Dirichlet boundary
condition at $c$ by $H^D_{(a,c)}$, i.e., $\dom(H^D_{(a,c)})$ consists
of all functions which are restrictions of functions from $\dom(H)$ to $(a,c)$
and satisfy $f(c)=0$.

\begin{lemma}[\cite{kst2}]\label{lem:pt}
The following properties are equivalent:
\begin{enumerate}
\item The spectrum of $H_{(a,c)}^D$ is purely discrete for some $c\in(a,b)$.
\item There is a real entire solution $\phi(z,x)$, which is
  nontrivial and lies in the domain of $H$ near $a$ for each
  $z\in\C$.
\item There exist real entire solutions $\phi(z,x)$, $\theta(z,x)$ with
  $W(\theta(z),\phi(z))\equiv 1$, such that $\phi(z,x)$
  lies in the domain of $H$ near $a$ for each $z\in\C$.
\end{enumerate}
\end{lemma}

Thus, for dealing with the singular Weyl function $M(z)$ as defined in
the Introduction it is necessary and sufficient that item (i)
holds. This will be our first hypothesis.

\begin{hypothesis}\label{hyp:gen}
  Suppose that the spectrum of $H^D_{(a,c)}$ is purely discrete for
  one (and hence for all) $c\in (a,b)$.
\end{hypothesis}

Note that this hypothesis is for example satisfied if $q(x) \to
+\infty$ as $x\to a$ (cf.\ Problem~9.7 in \cite{tschroe}).

\begin{remark}\label{rem:uniq}
  It is important to point out that a fundamental system satisfying
  the conditions we have imposed on $\phi(z,x)$ and $\theta(z,x)$ is
  not unique and any other such system is given by
\[
\ti{\theta}(z,x) = \E^{-g(z)} \theta(z,x) - f(z) \phi(z,x), \qquad
\ti{\phi}(z,x) = \E^{g(z)} \phi(z,x),
\]
where $f(z)$, $g(z)$ are entire functions with $f(z)$ real and $g(z)$
real modulo $\I\pi$.  The singular Weyl functions are related via
\[
\ti{M}(z) = \E^{-2g(z)} M(z) + \E^{-g(z)}f(z).
\]
\end{remark}

The singular Weyl function $M(z)$ is by construction analytic in
$\C\backslash\R$ and satisfies $M(z)=M(z^*)^*$.
Recall also from \cite[Lemma~3.3]{kst2} that associated with $M(z)$
there is a spectral measure $\rho$ given by the
Stieltjes--Liv\v{s}i\'{c} inversion formula \be\label{defrho}
\frac{1}{2} \left( \rho\big((x_0,x_1)\big) + \rho\big([x_0,x_1]\big)
\right)= \lim_{\eps\downarrow 0} \frac{1}{\pi} \int_{x_0}^{x_1}
\im\big(M(x+\I\eps)\big) dx.  \ee

In all assertions of this section Hypothesis~\ref{hyp:gen} is assumed.

\begin{theorem}[\cite{gz}]
  Define
  \begin{equation*}
  \hat{f}(\lam) = \lim_{c\uparrow b} \int_a^c \phi(\lam,x)
  f(x) dx,
  \end{equation*}
  where the right-hand side is to be understood as a limit in
  $L^2(\R,d\rho)$. Then the map
  \begin{equation}
    \label{eq:Udir}
 U: L^2(a,b) \to L^2(\R,d\rho),
  \qquad f \mapsto \hat{f},
  \end{equation}
  is unitary and its inverse is given by
  \begin{equation*}
f(x) = \lim_{r\to\infty} \int_{-r}^r \phi(\lam,x)
  \hat{f}(\lam) d\rho(\lam),
  \end{equation*}
  where again the right-hand side is to
  be understood as a limit in $L^2(a,b)$.  Moreover, $U$ maps $H$ to
  multiplication by $\lam$.
\end{theorem}

\begin{remark}
  We have seen in Remark~\ref{rem:uniq} that $M(z)$ is not
  unique. However, given $\ti{M}(z)$ as in Remark~\ref{rem:uniq}, the
  spectral measures are related by
\[
d\ti{\rho}(\lam) = \E^{-2g(\lam)} d\rho(\lam).
\]
Hence the measures are mutually absolutely continuous and the
associated spectral transformations just differ by a simple rescaling
with the positive function $\E^{-2g(\lam)}$.
\end{remark}

Finally, $M(z)$ can be reconstructed from $\rho$ up to an entire
function via the following integral representation.

\begin{theorem}[%[Theorem~4.1]
\cite{kst2}]\label{IntR}
  Let $M(z)$ be a singular Weyl function and $\rho$ its associated
  spectral measure. Then there exists an entire function $g(z)$ such
  that $g(\lam)\ge 0$ for $\lam\in\R$ and $\E^{-g(\lam)}\in L^2(\R,
  d\rho)$.

  Moreover, for any entire function $\hat{g}(z)$ such that
  $\hat{g}(\lam)>0$ for $\lam\in\R$ and $(1+\lam^2)^{-1}
  \hat{g}(\lam)^{-1}\in L^1(\R, d\rho)$ (e.g.\
  $\hat{g}(z)=\E^{2g(z)}$) we have the integral representation
  \begin{equation*}
     M(z) = E(z) + \hat{g}(z) \int_\R
  \left(\frac{1}{\lam-z} - \frac{\lam}{1+\lam^2}\right)
  \frac{d\rho(\lam)}{\hat{g}(\lam)}, \qquad z\in\C\backslash\sig(H),
  \end{equation*}
  where $E(z)$ is a real entire function.
\end{theorem}

As a consequence one obtains a criterion when the singular Weyl
function is a generalized Nevanlinna function with no nonreal poles
and the only generalized pole of nonpositive type at $\infty$.  We
will denote the set of all such generalized Nevanlinna functions by
$N_\kappa^\infty$ (see Appendix~C \cite{kst2} for a definition and
further references).

\begin{theorem}[%[Theorem~4.3]
\cite{kst2}]\label{thm:nkap}
  Fix the solution $\phi(z,x)$. Then there exists a corresponding solution
  $\theta(z,x)$ such that $M(z)\in N_\kappa^\infty$ for some
  $\kappa\le k$ if and only if $(1+\lam^2)^{-k-1} \in
  L^1(\R,d\rho)$. Moreover, $\kappa=k$ if $k=0$ or $(1+\lam^2)^{-k}
  \not\in L^1(\R,d\rho)$.
\end{theorem}

In order to identify possible values of $k$ one can try to bound
$\lam^{-k}$ by a linear combination of $\phi(\lam,x)^2$ and
$\phi'(\lam,x)^2$ which are in $L^1(\R, (1+\lam^2)^{-1}d\rho)$ by
\cite[Lemma~3.6]{kst2}.

\begin{remark}\label{rem:herg}
  Choosing a real entire function $g(z)$ such that $\exp(-2g(\lam))$
  is integrable with respect to $d\rho$, we see that
  \begin{equation*}
    M(z) =
  \E^{2g(z)} \int_\R \frac{1}{\lam-z} \E^{-2g(\lam)}d\rho(\lam) -
  E(z).
  \end{equation*}
  Hence if we choose $f(z) = \exp(-g(z)) E(z)$ and switch
  to a new system of solutions as in Remark~\ref{rem:uniq}, we see
  that the new singular Weyl function is a Herglotz--Nevanlinna
  function
  \begin{equation*}
    \ti{M}(z) = \int_\R \frac{1}{\lam-z}
  \E^{-2g(\lam)}d\rho(\lam).
  \end{equation*}
\end{remark}

\section{$N$-entire operators and de Branges spaces}
\label{sec:n-entire}
In a separable Hilbert space, let $A$ be a closed symmetric operator
with deficiency indices $(1,1)$ such that, for every $z\in\C$, there
is a positive constant $c_z$ for which
\begin{equation}
  \label{eq:regular}
  \norm{(A-zI)f}\ge c_z\norm{f},\quad\forall f\in\dom(A).
\end{equation}
The operator $A$ is said to be $n$-entire ($n\in\Z^+$) if moreover
 there exist
$n+1$ vectors $\mu_0,\ldots,\mu_n\in\cH$ such that
\begin{equation}
\label{eq:n-entire}
\cH = \ran(A-zI)\dotplus\lspan\{\mu_0+z\mu_1+\cdots+z^n\mu_n\}
\end{equation}
for all $z\in\C$ \cite{st1}. The operator $A$ is
called minimal $n$-entire whenever there is no smaller $n$ with
this property. Notice that a necessary but not sufficient condition
for $A$ to be minimal $n$-entire is that $\mu_n\ne 0$.

To any closed symmetric operator $A$ with deficiency indices $(1,1)$
satisfying (\ref{eq:regular}), there corresponds a de Branges space
$\cB_A$ in which the operator becomes multiplication by the
independent variable \cite{st1}. Recall that a de Branges Hilbert
space is a linear manifold given by
\begin{equation}
  \label{eq:dB}
 \cB:= \{F(z)\text{ entire}: F(z)/E(z), F^\#(z)/E(z)\in H^2(\C^+)\}
\end{equation}
with the inner product
\begin{equation}
  \label{eq:dB-inner-product}
  \inner{G}{F}_{\cB}:=\frac{1}{\pi}\int_\R\frac{G^*(x)F(x)}{\abs{E(x)}^2}dx\,,
\end{equation}
where $E(z)$ is in the Hermite--Biehler class, that is, an entire
function such that
$\abs{E(z)}>\abs{E(z^*)}$ for all $z\in\C^+$. Above we have used the
notation $F^\#(z):=F(z^*)^*$.
A de Branges space is a reproducing kernel Hilbert space with
reproducing kernel
\begin{equation*}
K(z,w) = \begin{cases}
		 \dfrac{E^\#(z)E(w^*) - E(z)E^\#(w^*)}{2\pi\I(z-w^*)},
					&w\ne z^*,
		 \\[2mm]
		 \frac{1}{2\pi\I}\left[E^{\#'}(z)E(z) - E'(z)E^\#(z)\right],
					&w=z^*.
		 \end{cases}
\end{equation*}
For further details we refer to \cite{debranges}.

Given $n\in\Z^+$, the set of $n$-associated functions is defined by
\begin{equation}
  \label{eq:assoc}
  \assoc_n(\cB):=\cB +z\cB+\dots+z^n\cB\,.
\end{equation}
Clearly, $\assoc_n(\cB)\subset\assoc_{n+1}(\cB)$ for $n\in\Z^+$. Also,
it is straightforward to verify that $E(z)$ is in $\assoc_1(\cB)$ but
not in $\cB$ \cite{debranges}.

\begin{lemma}[\cite{st1}]
  \label{lem:entire-n-associated}
  The operator $A$ is $n$-entire if and only if $\assoc_n(\cB_A)$
  contains a zero-free function. Moreover, $A$ is minimal $n$-entire
  if and only if, additionally, no zero-free function lies in
  $\assoc_m(\cB_A)$ for every $m<n$.
\end{lemma}

\begin{theorem}[\cite{st1}]
\label{thm:n-entire}
The following statements are equivalent:
\begin{enumerate}
\item The operator $A$ is $n$-entire.
\item Let $A_{\beta_1}$ and $A_{\beta_2}$, $\beta_1\ne \beta_2$,
	be 	canonical self-adjoint extensions of $A$ (that is, self-adjoint 
	restrictions of $A^*$). Set
	$\{x_j\}_{j\in\mathbb{N}}
	=\{x_j^+\}_{j\in\mathbb{N}}\cup\{x_j^-\}_{j\in\mathbb{N}}
	=\sigma(A_{\beta_1})$, where $\{x_j^+\}_{j\in\mathbb{N}}$ and
  	$\{x_j^-\}_{j\in\mathbb{N}}$ are the sequences of positive,
  	respectively nonpositive, elements of $\sigma(A_{\beta_1})$,
  	arranged according to increasing modulus. Then the following
  	assertions hold true:
	\begin{enumerate}[(C1)]
	\item The limit
		$\displaystyle{\lim_{r\to\infty}\sum_{0<|x_j|\le r}
		\frac{1}{x_j}}$
		exists.
	\item $\displaystyle{\lim_{j\to\infty}\frac{j}{x_j^{+}}
		=- \lim_{j\to\infty}\frac{j}{x_j^{-}}<\infty}$.
	\item Assuming that $\{b_j\}_{n\in\mathbb{N}}=\sigma(A_\beta)$,
		define\\[3mm]
		$\displaystyle
		h_\beta(z):=\left\{\begin{array}{ll}
			\displaystyle{\lim_{r\to\infty}\prod_{|b_j|\le r}
			\left(1-\frac{z}{b_j}\right)}
				& \mbox{ if }0\not\in\sigma(A_\beta),
			\\
			\displaystyle{z\lim_{r\to\infty}\prod_{0<|b_j|\le r}
			\left(1-\frac{z}{b_j}\right)}
				& \mbox{ otherwise. }
			   \end{array}\right.
		$\\[2mm]
		The series
		$\displaystyle{
		\sum_{x_j\ne 0}\abs{\frac{1}
		{x_j^{2n}h_{\beta_2}(x_j)h_{\beta_1}'(x_j)}}}$ is
		convergent.
\end{enumerate}
\end{enumerate}
\end{theorem}

\section{Main results}
\label{sec:main-results}

We begin this section by introducing and discussing a hypothesis that
is used to obtain the auxiliary results leading to the main ones.

\begin{hypothesis}
\label{hyp:main}
\begin{enumerate}
\item Suppose $H$ has purely discrete spectrum and let $\phi(z,x)$,
  $\chi(z,x)$ be entire solutions such that $\phi(z,x)$ is in the
  domain of $H$ near $a$ and $\chi(z,x)$ is in the domain of $H$ near
  $b$. Abbreviate
	\begin{equation}
	\label{eq:entire-wronskian}
	W(z) := W(\phi(z),\chi(z))
	\end{equation}
	which is of course also entire.

\item For every compact subset $K$ of $\C\times\rho(H)$, there exists $
	F\in L^1(a,b)$ such that
	\begin{equation*}
	\abs{\phi(w,x)\chi(z,x)} \le F(x)
	\end{equation*}
	for every $(w,z)\in K$.
\item We have
	\begin{equation*}
	\lim_{x\downarrow a} W_x(\phi(w), \chi(z)) = W(z)  \quad\text{and}\quad
	\lim_{x\uparrow b} W_x(\phi(w), \chi(z)) = W(w),
	\end{equation*}
	where the Wronskian here depends on $x$ since $\phi(w,x)$ and
	$\chi(z,x)$ are solutions of equations with different spectral parameters.
\end{enumerate}
\end{hypothesis}

Item (i) above amounts to assume that $M(z)$, hence the Weyl solution
$\psi(z,x)$, is a meromorphic function with (necessarily simple) poles
at $\sigma(H)$. Thus, given an entire function $W(z)$ whose zero set
includes $\sigma(H)$,
\begin{equation}
\label{eq:chips}
\chi(z,x)= W(z) \psi(z,x)
\end{equation}
is the entire solution that obeys \eqref{eq:entire-wronskian}. Note by
the way that this item implies Hypothesis~\ref{hyp:gen}.

The reason for Hypothesis~\ref{hyp:main} --- in particular, items (ii)
and (iii) --- will become clear later. For now we just point out that
it holds for example if both endpoints are in the limit circle case as
can be seen from Appendix~A in \cite{kst2} (for item (ii) see the
proof of Lemma~A.3 and for (iii) use $\chi(z,x)=
W(\chi(z),\phi(z))\theta(z,x) - W(\chi(z),\theta(z)) \phi(z,x)$ plus
Corollary~A.4). In the general case we first show that the items in
the above hypothesis are not independent. Our first result exploits
the fact that
\begin{equation*}
G(z,x,x) =\frac{\phi(z,x)\chi(z,x)}{W(z)}
\end{equation*}
is the diagonal of the kernel of the resolvent of $H$.

\begin{lemma}
  Consider the condition
\begin{equation}
\label{eq:hyp}
\int_a^b |\phi(z,x)\chi(z,x)| dx < \infty
\end{equation}
for some $z\in \C$.
\begin{enumerate}
\item Assume that $H$ is bounded from below. The inequality \eqref{eq:hyp}
  holds for one (and hence for all) $z< \inf\sig(H)$ if and only if
  $(H-z)^{-1}$ is trace class. In this case we have
\begin{equation*}
\frac{1}{W(z)}
  \int_a^b \phi(z,x) \chi(z,x) dx = \tr\big( (H-z)^{-1} \big), \qquad
  z\in\rho(H).
\end{equation*}
\item If \eqref{eq:hyp} holds for one $z\in\C\setminus\R$ then it
  holds for all $z\in\rho(H)$ and $(H-z)^{-1}$ is Hilbert--Schmidt.
\end{enumerate}
\end{lemma}

\begin{proof}
  (i) For $z< \inf\sig(H)$ the resolvent is a positive operator and
  hence the claim follows from the lemma on page~65 in
  \cite[Section~XI.4]{RS2}. Conversely, if $(H-z)^{-1}$ is trace
  class, then the above equality holds for all $z\in\rho(H)$ by
  Theorem~3.1 from \cite{br}.

  (ii) That $(H-z)^{-1}$ is Hilbert--Schmidt follows from the proof of
  Lemma 9.12 in \cite{tschroe}. The rest follows from the first
  resolvent formula which implies
\[
G(z,x,y) - G(w,x,y) = (z-w) \int_a^b G(z,x,t) G(w,t,y) dt.\qedhere
\]
\end{proof}

\begin{corollary}
  Assume that $H$ is bounded from below. Then assertion \eqref{eq:hyp}
  holds for one (and hence for all) $z< \inf\sig(H)$ if and only if
  $\sigma(H)$ obeys condition (C1) of Theorem~\ref{thm:n-entire}.
\end{corollary}

The Lagrange identity implies
\[
(w-z) \int_c^d \phi(w,x) \chi(z,x) dx = W_c(\phi(w,x), \chi(z,x))
- W_d(\phi(w,x), \chi(z,x))
\]
for arbitrary $a<c<d<b$. By item (ii) of Hypothesis~\ref{hyp:main}, we can take
limits $c\downarrow a$ and $d\uparrow b$ to obtain
\[
(w-z) \int_a^b \phi(w,x) \chi(z,x) dx = W_a(\phi(w,x),
\chi(z,x)) - W_b(\phi(w,x), \chi(z,x))
\]
with both limiting Wronskians being entire functions of both $z$ and
$w$. Moreover, note that $W_a(\phi(w,x), \chi(z,x))$ has the same
zeros as $W(z)$, and also $W_b(\phi(w,x), \chi(z,x))$ has the same zeros
as $W(w)$. However, it is not immediate that there is always
equality and hence we have imposed item (iii) of Hypothesis~\ref{hyp:main} which
finally yields
\begin{equation}
\label{eq:mf1}
\int_a^b \phi(w,x) \chi(z,x) dx =
- \frac{W(z) - W(w)}{z-w}.
\end{equation}
In the limit $w\to z$ this
gives
\begin{equation*}
\int_a^b \phi(z,x) \chi(z,x) dx = - \frac{d}{dz}W(z).
\end{equation*}
Next we want to relate this to Weyl--Titchmarsh--Kodaira theory from
Section~\ref{sec:swm}. Of course $\chi(z,x)$ is related to the Weyl
solution via (\ref{eq:chips})
and we obtain the following formula which will be crucial for us.

\begin{lemma}\label{lem:mf}
 Assume Hypothesis~\ref{hyp:main} and abbreviate
\begin{equation*}
\psi^{(j)}(z,x) := \frac{\partial^j}{\partial z^j}
  \psi(z,x).
\end{equation*}
Then
\begin{multline}
\label{eq:mf2}
\frac{(w-z)^{j+1}}{j!}
  \int_a^b \phi(w,x) \psi^{(j)}(z,x) dx 
  \\
  = 1 - \sum_{k=0}^{j}
  \frac{(w-z)^k}{k!} W_b(\phi(w,x), \psi^{(k)}(z,x)).
\end{multline}
\end{lemma}

\begin{proof}
  The case $j=0$ follows from \eqref{eq:mf1} upon using
  \eqref{eq:chips}. The case $j\ge 1$ follows from induction by
  differentiation with respect to $z$. Note that by Cauchy's integral
  formula the derivatives w.r.t. $z$ of $\chi(z,x)$ also satisfy item (ii) of
  Hypothesis~\ref{hyp:main}.
\end{proof}

Note that when $\lambda$ is an eigenvalue we obtain:

\begin{corollary}
\label{cor:mf}
Assume Hypothesis~\ref{hyp:main}. If $\lam\in\sig(H)$ and
$z\in\rho(H)$, then
  \begin{equation}
  \label{eq:mf3}
  \int_a^b \phi(\lam,x) \psi^{(j)}(z,x) dx = \frac{j!}{(\lam-z)^{j+1}}.
  \end{equation}
\end{corollary}

\begin{proof}
  The assertion for $j=0$ follows from (\ref{eq:mf2}) by taking into
  account that $W_b(\phi(\lam,x), \psi(z,x))=0$ whenever
  $\lam\in\sig(H)$. Now use induction as before.
\end{proof}

The next assumption will allow us to associate $H$ with a certain symmetric non 
self-adjoint operator. Also, in combination with Hypothesis~\ref{hyp:gen},
it will imply item (i) of Hypothesis~\ref{hyp:main}.

\begin{hypothesis}
\label{hyp:limit-circle}
  The endpoint $b$ is in the limit circle case.
\end{hypothesis}

Let $A$ be the closure of the restriction of $H$ to functions
vanishing in a neighborhood of $b$. By
Hypotheses~\ref{hyp:gen} and \ref{hyp:limit-circle}, this operator has deficiency
indices $(1,1)$ and satisfies (\ref{eq:regular}). One way of
constructing the corresponding de Branges space $\cB_A$ is the
following. Fix two real-valued solutions $c(x)$ and $s(x)$ corresponding
to the same spectral parameter with $W(c,s)=1$. Now introduce the
entire function \be E(z)= W_b(c,\phi(z)) + \I W_b(s,\phi(z)).  \ee
Note that by our limit circle assumption the limit of the Wronskians
exist at $b$ and are indeed entire with respect to $z$ (cf.\
Appendix~A in \cite{kst2}). Moreover, an analogous computation as
before verifies
\begin{equation*}
\frac{E(z) E^\#(w^*) - E(w^*) E^\#(z)}{2\I
  (w^* -z)} = \int_a^b \phi(w,x)^* \phi(z,x) dx, \quad
w,z\in\C.
\end{equation*}
In particular, taking $w=z$ this shows that $E(z)$ is a
Hermite--Biehler function.  Moreover, note that $E(z)$ does not have
any real zero, since otherwise both, $W_b(c,\phi(z))$ and
$W_b(s,\phi(z))$ would vanish, contradicting $W(c,s)=1$. Now, $\cB_A$
is the de Branges space generated by $E(z)$ as specified in
(\ref{eq:dB}). The reproducing kernel of this space is
given by
\begin{equation}
\label{eq:reproducing-kernel}
K(w,z) = \int_a^b \phi(w,x)^* \phi(z,x) dx, \quad w,z\in\C.
\end{equation}
This also shows that the de Branges norm equals the spectral norm,
\begin{equation}
\label{eq:equality-norms}
\inner{F}{G}_{\cB_A} = \int_\R F(x)^*G(x) d\rho(x).
\end{equation}

\begin{remark}
\label{rem:various-facts-db}
\begin{enumerate}[(a)]
\item \label{item:db-equivalent-to-l2} Identities
  \eqref{eq:reproducing-kernel} and \eqref{eq:equality-norms} imply
  that $\cB_A$ and $L^2(\R,d\rho)$ are unitarily equivalent in the
  sense that, the restriction to $\R$ of every function in $\cB_A$
  belongs to $L^2(\R,d\rho)$ while for every function in
  $L^2(\R,d\rho)$ there exist one, and only one, function in $\cB_A$
  whose restriction to $\R$ belongs to the same equivalence class
  (with respect to the measure $\rho$).

\item \label{item:assoc_n-in-the-right-place} Since
  $\assoc_n\cB(E)=\cB(E_n)$ with $E_n(z):=(z+\I)^nE(z)$ (as sets)
  \cite{lw}, one easily obtains
  \[
  \assoc_n\cB(E) \cong
  L^2\left(\R,\tfrac{d\rho}{(x^2+1)^n}\right),
  \]
  where the isomorphism is in the sense given in (a).

\item \label{item:two-assoc-functions} Since
  $E(z)\in\assoc_1(\cB_A)\setminus\cB_A$ and $A$ is densely defined,
  the functions
  \begin{equation*}
    W_b(c,\phi(z)) = \frac{E(z)+E^\#(z)}{2}\quad
    \text{and}\quad
    W_b(s,\phi(z)) = \frac{E(z)-E^\#(z)}{2\I}
  \end{equation*}
  also belong to $\assoc_1(\cB_A)\setminus\cB_A$ \cite{debranges}.
\end{enumerate}
\end{remark}

\begin{theorem}
\label{thm:remaining-implication}
Assume Hypotheses~\ref{hyp:main} and \ref{hyp:limit-circle} and let
$A$ be the operator defined above. If there is $z\in\rho(H)$ such that
$\psi^{(n-1)}(z,x) \in L^2(a,b)$, then the operator $A$ is
$n$-entire.
\end{theorem}
\begin{proof}
  Our assumption implies, by letting $j:=n-1$ in \eqref{eq:mf2},
  that the left-hand side in \eqref{eq:mf2} is in
  $\assoc_n(\cB_A)$. The same is true for the sum on the right-hand
  side which is a sum of a polynomial in $w$ of degree $n-1$ times
  $W(c,\phi(w))$ and $W(s,\phi(w))$ since
  \begin{equation*}
    W_b(\phi(w),\psi^{(j)}(z)) =
  W_b(s,\psi^{(j)}(z)) W_b(c,\phi(w)) - W_b(c,\psi^{(j)}(z))
  W_b(s,\phi(w)).
  \end{equation*}
  In view of item (\ref{item:two-assoc-functions}) of
  Remark~\ref{rem:various-facts-db} one has that $1\in\assoc_n(\cB_A)$
  which in turn implies the assertion by
  Lemma~\ref{lem:entire-n-associated}.
\end{proof}
\begin{theorem}
\label{thm:main-equivalences}
Let the assumptions of Theorem~\ref{thm:remaining-implication} hold. Then,
the following are equivalent:
\begin{enumerate}
\item The operator $A$ is $n$-entire.

\item There is a choice of the entire solution $\theta(z,x)$ such that
  $M(z)\in N^\infty_\kappa$ for $\kappa\le n-1$.
\end{enumerate}
\end{theorem}

\begin{proof}
  (i) $\Rightarrow$ (ii). Since $A$ is $n$-entire, there exists a zero-free
  function in $\assoc_n(\cB_A)$. Without loss of generality we can
  assume this function to be equal to $1$. By item
  (\ref{item:assoc_n-in-the-right-place}) of
  Remark~\ref{rem:various-facts-db} one has that
  $(1+\lambda^2)^{-n}$ is in $L^1(\R,d\rho)$. Hence
  Theorem~\ref{thm:nkap} yields (ii).

  (ii)  $\Rightarrow$ (i). By Theorem~\ref{thm:nkap} (ii) implies
  that $(1+\lambda^2) \in L^1(\R,d\rho)$. The claim follows now from item
  (\ref{item:assoc_n-in-the-right-place}) of
  Remark~\ref{rem:various-facts-db} and Lemma~\ref{lem:entire-n-associated}.
\end{proof}

\begin{corollary}
  \label{cor:n-entire-weak}
Let the assumptions of Theorem~\ref{thm:remaining-implication}
hold. Suppose moreover that one of the following holds true:
\begin{enumerate}[(a)]
\item There is a choice of the entire solution $\theta(z,x)$ such that
  $M(z)\in N^\infty_\kappa$ for $\kappa\le n-1$.
\item There exists $z\in\rho(H)$ such that $\psi^{(n-1)}(z,x) \in L^2(a,b)$.
\end{enumerate}
Then there is another self-adjoint extension
$H'$ of $A$ such that $\sigma(H)$ and $\sigma(H')$ satisfy (C1), (C2),
(C3) of Theorem~\ref{thm:n-entire}.
\end{corollary}
\begin{proof}
The claim is obtained immediately from
Theorems~\ref{thm:remaining-implication} and \ref{thm:main-equivalences}
in combination with Theorem~\ref{thm:n-entire}.
\end{proof}
\begin{hypothesis}
\label{hyp:last}
  Let $\phi(z,x)$ and $\psi(z,x)$ be such that if
  \begin{equation*}
    \int_a^b\phi(\lam,x)\psi^{(j)}(z,x)dx\in L^2(\R,d\rho)
  \end{equation*}
 for some $j\in\N$ and $z\in\rho(H)$, then
  \begin{equation*}
   \psi^{(j)}(z,x) =\lim_{r\to\infty}\int_{-r}^r\phi(\lam,x)
  \left(\int_a^b\phi(\lam,y)\psi^{(j)}(z,y)dy\right)d\rho(\lam)\,,
  \end{equation*}
 where the limit is understood as a limit in $L^2(a,b)$.
\end{hypothesis}
\begin{theorem}
\label{thm:main-equivalences-true}
  Let the assumptions of Theorem~\ref{thm:remaining-implication} hold and
  assume Hypothesis~\ref{hyp:last}. Then, the
  following are equivalent:
\begin{enumerate}
\item The operator $A$ is $n$-entire.

\item There is a choice of the entire solution $\theta(z,x)$ such that
  $M(z)\in N^\infty_\kappa$ for $\kappa\le n-1$.

\item $\psi^{(n-1)}(z,x) \in L^2(a,b)$ for one (and hence for all)
  $z\in\rho(H)$.
\end{enumerate}
\end{theorem}
\begin{proof}
  In view of Theorems~\ref{thm:remaining-implication} and
  \ref{thm:main-equivalences}, one only has to show that (ii)
  $\Rightarrow$ (iii). By Theorem~\ref{thm:nkap}, (ii) implies that
  $(1+\lambda^2)^{-n}$ is in $L^1(\R,d\rho)$, so the function
  $(n-1)!(\lambda-z)^{-n}$ is in $L^2(\R,d\rho)$. Therefore there is a
  function $\eta(z,x)\in L^2(a,b)$ such that
  $\eta(z,x)=U^{-1}\left(\frac{(n-1)!}{(\lambda-z)^n}\right)$. By
  Corollary~\ref{cor:mf}, Hypothesis~\ref{hyp:last} implies that, at
  least for one $z\in\rho(H)$, $\eta(z,x)=\psi^{(n-1)}(z,x)$.
\end{proof}
The proof of the previous assertion can be complemented to obtain the
following sharpened version of it.

\begin{theorem}
\label{thm:main-equivalence-complete}
  Under the assumptions of Theorem~\ref{thm:main-equivalences-true}, the
  following are equivalent:
\begin{enumerate}
\item The operator $A$ is minimal $n$-entire.

\item There is a choice of the entire solution $\theta(z,x)$ such that
  $M(z)\in N_{n-1}^\infty$.

\item $\psi^{(n-1)}(z,x) \in L^2(a,b)$ but
  $\psi^{(n-2)}(z,x)\not\in L^2(a,b)$, for one (and hence for all)
  $z\in\rho(H)$.
\end{enumerate}
\end{theorem}

A class of operators attracting attention nowadays and for which
Hypotheses \ref{hyp:main}, \ref{hyp:limit-circle}, and \ref{hyp:last}
are satisfied is the class of spherical Schr\"odinger operators.

\begin{theorem}\label{thm:bessel}
Fix $l\ge -\frac{1}{2}$ and $b>0$. Suppose
\be\label{defHBes}
\tau = -\frac{d^2}{dx^2} + \frac{l(l+1)}{x^2} + q(x), \quad x\in(0,b),
\ee
where
\be
\begin{cases}
x q(x) \in L^1(0,b), &  l> -\frac{1}{2},\\
x(1-\log(x/b)) q(x) \in L^1(0,b), & l=-\frac{1}{2}.\end{cases}
\ee
If $\tau$ is limit circle at $a=0$ we impose the usual boundary condition
(corresponding to the Friedrichs extension; see also \cite{bg}, \cite{ek})
\be
\lim_{x\to0} x^l ( (l+1)f(x) - x f'(x))=0, \qquad l\in[-\frac{1}{2},\frac{1}{2}).
\ee
Then the assumptions of Theorem~\ref{thm:main-equivalences-true} are satisfied
and, whenever $n\in\Z^+$ obeys $2n\ge\floor{l+\frac52}$ (equivalently,
$n>\frac{l}{2}+\frac34$), the corresponding operator $A$ is $n$-entire.
\end{theorem}

\begin{proof}
  Item (ii) of Hypothesis~\ref{hyp:main} follows from Lemma 2.2 and
  2.6 in \cite{kst} and item (iii) follows from Corollary~3.12 in
  \cite{kt}. The first part of Lemma 4.4 in \cite{kt} implies that
  Hypothesis~\ref{hyp:last} is satisfied.  Moreover, that
  $\psi^{(n-1)}(z,x) \in L^2(a,b)$ for the proposed values of $n$ is
  shown in Lemma 4.4 of \cite{kt}.
\end{proof}

In particular, this generalizes Theorem~4.3 from \cite{st2}. Note that we
could even allow a nonintegrable singularity at $b$ as long as $\tau$ is
limit circle at $b$. Of course this also generalizes Corollary~4.4 from
\cite{st2}:

\begin{corollary}\label{cor:bessel}
Under the assumptions of Theorem~\ref{thm:bessel},
the spectra of two canonical self-adjoint extensions $H_1$, $H_2$ of $A$
satisfy conditions (C1), (C2) and (C3) of Theorem~\ref{thm:n-entire} whenever
$2n\ge\floor{l+\frac52}$.
\end{corollary}

Finally, one has the following consequence of
Theorem~\ref{thm:main-equivalence-complete}.

\begin{corollary}
  \label{cor:bessel-sharp}
  Under the assumptions of Theorem~\ref{thm:bessel}, the underlying
  operator $A$ is minimal $\floor{l+\frac52}$-entire.
\end{corollary}

\begin{acknowledgments}
G.T. and J.H.T. gratefully acknowledge the kind hospitality of the {\em Instituto
de Investigaciones en Matem\'aticas Aplicadas y en Sistemas (IIMAS)},
(Mexico City, MX) during a visit in 2013 where
most of this article was written.
\end{acknowledgments}

\end{document}